\documentclass[11pt]{amsart}
\usepackage{graphicx}
\usepackage{amscd}
\usepackage{amsmath}
\usepackage{amsxtra}
\usepackage{amsfonts}
\usepackage{amssymb}
\usepackage{xcolor}
\usepackage{mathrsfs}  
\usepackage{leftindex}
\usepackage[euler]{textgreek}

\newtheorem{theorem}{Theorem}[section]

\newtheorem{lemma}[theorem]{Lemma}
\newtheorem{proposition}[theorem]{Proposition}

\newtheorem{conjecture}[theorem]{Conjecture}
\theoremstyle{definition}
\newtheorem{definition}[theorem]{Definition}

\theoremstyle{remark}

\renewcommand{\theclaim}{\textup{\theclaim}}

\numberwithin{equation}{section}

\def\openone

{\mathchoice
	
	{\hbox{\upshape \small1\kern-3.3pt\normalsize1}}
	
	{\hbox{\upshape \small1\kern-3.3pt\normalsize1}}
	
	{\hbox{\upshape \tiny1\kern-2.3pt\SMALL1}}
	
	{\hbox{\upshape \Tiny1\kern-2pt\tiny1}}}

\makeatletter

\newbox\ipbox

\newcommand{\ip}[2]{\left\langle #1\, , \,#2\right\rangle}
\newcommand{\diracb}[1]{\left\langle #1\mathrel{\mathchoice
		
		{\setbox\ipbox=\hbox{$\displaystyle \left\langle\mathstrut
				#1\right.$}
			
			\vrule height\ht\ipbox width0.25pt depth\dp\ipbox}
		
		{\setbox\ipbox=\hbox{$\textstyle \left\langle\mathstrut
				#1\right.$}
			
			\vrule height\ht\ipbox width0.25pt depth\dp\ipbox}
		
		{\setbox\ipbox=\hbox{$\scriptstyle \left\langle\mathstrut
				#1\right.$}
			
			\vrule height\ht\ipbox width0.25pt depth\dp\ipbox}
		
		{\setbox\ipbox=\hbox{$\scriptscriptstyle \left\langle\mathstrut
				#1\right.$}
			
			\vrule height\ht\ipbox width0.25pt depth\dp\ipbox}
		
	}\right. }

\newcommand{\dirack}[1]{\left. \mathrel{\mathchoice
		
		{\setbox\ipbox=\hbox{$\displaystyle \left.\mathstrut
				#1\right\rangle$}
			
			\vrule height\ht\ipbox width0.25pt depth\dp\ipbox}
		
		{\setbox\ipbox=\hbox{$\textstyle \left.\mathstrut
				#1\right\rangle$}
			
			\vrule height\ht\ipbox width0.25pt depth\dp\ipbox}
		
		{\setbox\ipbox=\hbox{$\scriptstyle \left.\mathstrut
				#1\right\rangle$}
			
			\vrule height\ht\ipbox width0.25pt depth\dp\ipbox}
		
		{\setbox\ipbox=\hbox{$\scriptscriptstyle \left.\mathstrut
				#1\right\rangle$}
			
			\vrule height\ht\ipbox width0.25pt depth\dp\ipbox}
		
	} #1\right\rangle}

\newcommand{\beq}{\begin{equation}}
	
	\newcommand{\eeq}{\end{equation}}

\newcommand{\cj}[1]{\overline{#1}}

\newcommand{\bz}{\mathbb{Z}}

\newcommand{\br}{\mathbb{R}}

\newcommand{\bn}{\mathbb{N}}

\def\blfootnote{\xdef\@thefnmark{}\@footnotetext}

\renewcommand{\mod}{\operatorname{mod}}

\input xy
\xyoption{all}
\usepackage{amssymb}

\newcommand{\Per}{\operatorname*{Per}}

\def\-{^{-1}}

\def\ty{\emptyset}

\usepackage[OT2,T1]{fontenc}
\DeclareSymbolFont{cyrletters}{OT2}{wncyr}{m}{n}
\DeclareMathSymbol{\Sha}{\mathalpha}{cyrletters}{"58}
	
\newcommand{\Leb}{\operatorname*{Leb}}

\begin{document}
	
	\title[Some Plancherel identities for unbounded subsets of $\br$ in duality]{Some Plancherel identities for unbounded subsets of $\br$ in duality}

	\author{Piyali Chakraborty}
	\address{[Piyali Chakraborty] University of Central Florida\\
		Department of Mathematics\\
		4000 Central Florida Blvd.\\
		P.O. Box 161364\\
		Orlando, FL 32816-1364\\
		U.S.A.\\} \email{Piyali.Chakraborty@ucf.edu}

	\author{Dorin Ervin Dutkay}
	\address{[Dorin Ervin Dutkay] University of Central Florida\\
		Department of Mathematics\\
		4000 Central Florida Blvd.\\
		P.O. Box 161364\\
		Orlando, FL 32816-1364\\
		U.S.A.\\} \email{Dorin.Dutkay@ucf.edu}


	\begin{abstract}
		In relation to Fuglede's conjecture, we establish several Plancherel-type identities and demonstrate the surjectivity of the Fourier transform between certain unbounded tiling sets of $\mathbb{R}$ that are in duality. In the terminology commonly used in the context of Fuglede's conjecture, our result states that an open set tiles $\mathbb{R}$ by the finite set $\{0,1,\dots,p-1\}$ if and only if it admits a spectrum (or, equivalently, a dual pair measure) given by the Lebesgue measure on $\left[-\tfrac{1}{2p}, \tfrac{1}{2p}\right] + \mathbb{Z}$.

	\end{abstract}
	\maketitle
	\tableofcontents
	\newcommand{\Ds}{\mathsf{D}}
	\newcommand{\Dmax}{\mathscr D_{\operatorname*{max}}}
	\newcommand{\Dmin}{\Ds_{\operatorname*{min}}}
	\newcommand{\dom}{\operatorname*{dom}}

	\section{Introduction}
	In 1958, Irving Segal posed the following question to Bent Fuglede:  
	Let $\Omega$ be an open subset of $\br^d$. Consider the partial differential operators  
	$$D_1=\frac1{2\pi i}\frac\partial{\partial x_1},\dots,D_d=\frac1{2\pi i}\frac\partial{\partial x_d},$$  
	defined on the space $C_0^\infty(\Omega)$ of smooth, compactly supported functions on $\Omega$. Under what conditions do these differential operators admit commuting (unbounded) self-adjoint extensions $H_1,\dots,H_d$ on $L^2(\Omega)$, where commutation is understood in the sense of their spectral measures?
	
	In his seminal 1974 paper \cite{Fug74}, Fuglede provided an answer to Segal’s question in the case where $\Omega$ is a \textit{connected}, \textit{finite-measure}, \textit{Nikodym domain}—that is, when the Poincar\'e inequality holds:
	
	\begin{theorem}\label{th2.1}\cite[Theorem I]{Fug74}
		Let $\Omega\subset\br^d$ be an open, connected Nikodym region of finite measure. Define  
		$$e_\lambda(x)=e^{2\pi i\lambda\cdot x},\quad (x\in\br^d,\lambda\in\br^d).$$
		There exist commuting self-adjoint extensions of the differential operators $\{D_j: j=1,\dots,d\}$ if and only if there exists a subset $\Lambda$ of $\br^d$ such that the family of exponential functions
		$$\{e_\lambda :\lambda\in\Lambda\}$$
		forms an orthogonal basis for $L^2(\Omega)$.
		
	\end{theorem}
	
	\begin{definition}
		\label{defsp}
		A measurable set $\Omega\subset\br^d$ of finite Lebesgue measure is called \textit{spectral} if there exists a set $\Lambda\subset\br^d$ such that the family of exponential functions  
		$$\{e_\lambda : \lambda\in\Lambda\}$$  
		forms an orthogonal basis for $L^2(\Omega)$. In this case, $\Lambda$ is called a \textit{spectrum} for $\Omega$.
		
		We say that $\Omega$ \textit{tiles} $\br^d$ by translations if there exists a set $\mathcal T\subset\br^d$ such that the translates $\{\Omega+t : t\in\mathcal T\}$ form a partition of $\br^d$ up to measure zero. The set $\mathcal T$ is then called a \textit{tiling set} for $\Omega$, and we say that $\Omega$ \textit{tiles $\br^d$ by} $\mathcal T$.
	\end{definition}
	
	Thus, in this framework, when $\Omega$ is connected, has finite measure, and is a Nikodym domain, there exist commuting self-adjoint extensions of the differential operators $\{D_j\}$ if and only if $\Omega$ is spectral.  
	Because this characterization is somewhat abstract, Fuglede proposed his famous conjecture:
	
	\begin{conjecture}{\bf[Fuglede's Conjecture]}
		\label{conFu}
		A measurable subset $\Omega\subset\br^d$ of finite Lebesgue measure is spectral \textit{if and only if} it tiles $\br^d$ by translations.
	\end{conjecture}
	
	The conjecture was later shown to be false in dimensions $d\ge3$ \cite{Tao04, FMM06}, but it remains true for convex domains \cite{LM22}.
	
	In \cite{Ped87}, Steen Pedersen extended Fuglede’s result by removing both the Nikodym condition and the finite-measure assumption. Naturally, the definition of a spectral set had to be adapted for infinite-measure domains (since $e_\lambda$ is not square-integrable in that case).
	
	\begin{definition}\label{defp1}
		For a function $f\in L^1(\br^d)$, define its classical Fourier transform as  
		$$\hat f(x)=\int_{\br^d}f(x)e^{-2\pi it\cdot x}\,dx,\quad(t\in\br^d).$$
		
		Let $\Omega\subset\br^d$ be measurable and let $\mu$ be a positive Radon measure on $\br^d$. We say that $(\Omega,\mu)$ is a \textit{spectral pair} if:  
		(1) for each $f\in L^1(\Omega)\cap L^2(\Omega)$, the continuous function $t\mapsto \hat f(t)$ satisfies $\int|\hat f|^2\,d\mu<\infty$; and  
		(2) the mapping $f\mapsto \hat f$ from $L^1(\Omega)\cap L^2(\Omega)\subset L^2(\Omega)$ into $L^2(\mu)$ is isometric and has dense range.
		
		This map then extends by continuity to an isometric isomorphism  
		$$\mathscr F:L^2(\Omega)\to L^2(\mu).$$
		
		The isometry property means that we have a Plancherel identity implemented by taking the Fourier transform of functions in $L^2(\Omega)$ and restricting it to the support of the measure $\mu$ and integrating the absolute value squared of this restriction of the Fourier transform against the measure $\mu$.
		
		The set $\Omega$ is called \textit{spectral} if there exists a measure $\mu$ such that $(\Omega,\mu)$ forms a spectral pair; $\mu$ is then called a \textit{pair measure} or a {\it dual measure} for $\Omega$. When $\Omega$ has finite measure, this definition coincides with the earlier one \cite[Corollary 1.11]{Ped87}, and the dual measure $\mu$ is the counting measure on the spectrum $\Lambda$.
	\end{definition}
	
	With this broader definition, Pedersen generalized Fuglede’s theorem by eliminating both the Nikodym and finite-measure conditions, while maintaining the assumption that $\Omega$ is connected.
	
	\begin{theorem}\label{thp5}\cite[Theorem 2.2]{Ped87}
		Let $\Omega$ be an open, connected subset of $\br^d$. Then the partial differential operators $\{D_j\}$ admit commuting self-adjoint extensions if and only if $\Omega$ is a spectral set.
	\end{theorem}
	
 	The results were then extended to disconnected domains in \cite{CD25}. In this case, if $\Omega$ is spectral, then there are commuting self-adjoint extensions of the differential operators $\{D_j\}$. However, the converse has to be adjusted a little: the Fourier transform in Definition \ref{defp1} is allowed to have some weights which differ on the components of $\Omega$, and the resulting condition is necessary and sufficient.

In the 1974 paper Fuglede proved that his conjecture is true when the spectrum or the tiling set is a lattice in $\br^d$.

\begin{theorem}\label{thf1}
	
	Let $\Omega$ be a measurable subset of $\br^d$ of finite measure and let $A$ be some invertible $d\times d$ matrix. Then $\Omega$ tiles $\br^d$ by the lattice $A\bz^d$ if and only if $\Omega$ has spectrum the dual lattice $(A^T)^{-1}\bz^d$.
\end{theorem}

Many examples of spectral sets are known in the case when $\Omega$ has finite measure, and the spectrum can be a non-lattice, see e.g, \cite{Kol24}. However, much less is known for sets of infinite measure. Of course, the simplest, and most important example of a spectral set of infinite measure is $\br^d$, with pair measure the Lebesgue measure on $\br^d$. To obtain some other examples, a  result similar to Theorem \ref{thf1} was proved for sublattices of $\br^d$ in \cite{CD25}:

\begin{theorem}\label{th5.9}
	Write $\br^d=\br^{d_1}\times\br^{d_2}$. Let $A$ be an invertible real $d\times d$ matrix. The set $\Omega$ tiles with the discrete subgroup $\mathcal T=A(\mathbb{Z}^{d_{1}}\times\{0\})$ if and only if $\Omega$ has pair measure $A^T(\Sha_{\mathbb{Z}^{d_{1}}}\times \mathfrak m_{d_{2}})$ which is supported on the dual set $\mathcal T^*$,
		$$\mathcal T^*:=\{\gamma^*\in\br^d: \gamma^*\cdot \gamma\in\bz \text{ for all }\gamma\in\mathcal T\}.$$
\end{theorem}		
		Here the Dirac comb $\Sha_{\bz^{d_1}}$ is the counting measure on $\bz^{d_1}$, and $\mathfrak m_{d_2}$ is the Lebesgue measure on $\br^{d_2}$.
		
			For a Borel measure $\mu$ on $\br^d$, the measure $A^T\mu$ is defined by $A^T\mu(E)=|\det(A^T)^{-1}|\mu(A^TE)$ for any Borel subset $E$ in $\br^d$. Equivalently, for any compactly supported continuous function $f$ on $\br^d$:
		\begin{equation}
			\label{eqdil1}
			\int f\, dA^T\mu=|\det(A^T)^{-1}|\int f((A^T)^{-1}x)\,d\mu(x).
		\end{equation}
		
		Note that if the measure $\mu$ is supported on a set $\Lambda$, then $A^T\mu$ is supported on $(A^T)^{-1}\Lambda$.

	In this paper we provide a new class of examples of unbounded spectral sets in $\br$ and we prove the following theorem:
	
	\begin{theorem}\label{thm}
		Let $\Omega$ be an open subset in $\br$ and let $p\in\bn$, $p\geq 2$. Then $\Omega$ tiles $\br$ by $\{0,1,\dots,p-1\}$ if and only if $\Omega$ is a spectral set with pair measure $\mu$ equal to a renormalized Lebesgue measure on $\left[-\frac1{2p},\frac{1}{2p}\right]+\bz$, $\mu=p\Leb_{\left[-\frac1{2p},\frac1{2p}\right]+\bz}$ .
		\end{theorem}

		\section{Proof of Theorem \ref{thm}}
	
	We will need some definitions and notations.

	\begin{definition}\label{tile}
	Let $\Omega_0$ and $\Omega$ be two measurable subsets in $\br$, and let $\mathcal T$ be a finite or countable subset of $\br$. We say that $\Omega_0$ {\it tiles}  $\Omega$ by $\mathcal T$ if $\{\Omega_0+t : t\in\mathcal T\}$ is a partition of $\Omega$, up to measure zero. 
	
	We denote by $T_a$ the translation operator on functions on $\br$ 
	$$(T_a)f(x)=f(x-a),\quad (x\in\br,a\in\br).$$
	\end{definition}

	\begin{definition}
		\label{congset}
		Let $E$ and $F$ be two measurable subsets of $\br$. We say that $E$ and $F$ are {\it congruent} $\mod \bz$ if there exists a partition (up to measure zero) of $E$ into possibly empty sets $\{E_k : k\in\bz\}$ such that $\{E_k+k : k\in\bz\}$ is a partition of $F$ (up to measure zero). 
	\end{definition}
	
	\begin{definition}
		\label{spme}
		Let $\mu$ be a finite Radon measure on $\br$ and $\Lambda\subset\br$. We say that $\mu$ is a {\it spectral measure} with {\it spectrum} $\Lambda$ if the family $\{e_\lambda : \lambda\in\Lambda\}$ forms an orthogonal basis for $L^2(\mu)$.
	\end{definition}
	\begin{definition}\label{sha}
		
		For a measurable subset $E$ of $\br$, we denote by $|E|$ its Lebesgue measure and by $\chi_E$ its characteristic function. 
		For a finite subset $F$, we denote by $|F|$ its cardinality.		
		We denote by $\delta_\gamma$ the Dirac measure at the point $\gamma$ in $\br$. 		
		For a discrete subset $\Gamma$ of $\br$, we define the {\it Dirac comb} on $\Gamma$ to be the counting measure on $\Gamma$,
		$$\Sha_\Gamma=\sum_{\gamma\in\Gamma}\delta_\gamma.$$
		
		For a measurable subset $E$ of $\br$, we denote by $\mbox{Leb}_E$ the Lebesgue measure restricted to $E$. 
		
		For a function $f$ on $\br$ we denote its periodization by
		$$\Per(f)(t)=\sum_{k\in\bz}f(t+k),\quad (t\in\br).$$
	\end{definition}

	We begin now the proof of our main result.
	
	\begin{proof}[Proof of Theorem \ref{thm}]
	We assume that $\Omega$ tiles $\br$ by $\{0,1,\dots,p-1\}$.
	
	\begin{proposition}\label{pr1}
	Suppose $\Omega$ tiles $\br$ by $\{0,1,\dots,p-1\}$. 
	Define $\Omega_0=\Omega\cap[0,p)$. 
	\begin{enumerate}
		\item $\Omega+kp\subset\Omega$ for all $k\in\bz$. 
		\item $\Omega_0$ tiles $\Omega$ by $p\bz$; in particular $$\Omega=\Omega_0+p\bz.$$
		\item $\Omega_0$ tiles $\br$ by $\bz$. Therefore $\Omega_0$ has measure 1, has spectrum $\bz$, and is congruent to $[0,1)$ modulo $\bz$. 
	\end{enumerate}
	\end{proposition}
	
	\begin{proof}
		Let $x\in\Omega$. Then $x+p\in\Omega$. Indeed, since $\Omega$ tiles with $\{0,\dots,p-1\}$ we can write $x+p=y+j$ for $y\in\Omega$ and $j\in\{0,\dots,p-1\}$. If $j\neq 0$, then $x+(p-j)=y$ so $y\in \Omega\cap(\Omega+(p-j))$, which can happen only for $y$ in a measure zero set. 
		
		If $x-p=y+j$ with $x\in\Omega$ and $j\in\{0,1,\dots,p-1\}$ then $\Omega\ni x=y+p+j\in\Omega+j$ and therefore $j$ has to be $0$ and $x-p\in\Omega$. By induction, we get $x+pk\in\Omega$, for $k\in\bz$. This proves (i).

		For (ii), it is clear from (i) that $\Omega_0+p\bz\subseteq\Omega$. Now take $x\in\Omega$. Then we can write $x=y+pk$ for some $y\in[0,p)$ and $k\in\bz$. With (i), $y=x-kp\in\Omega$. Thus $y\in\Omega_0$ and $x\in\Omega_0+p\bz$. This proves the reverse inclusion. 
		
		Since $\Omega_0$ is contained in $[0,p)$, the sets $\Omega_0+pk$, $k\in\bz$, are mutually disjoint. This proves (ii). 
		
		For (iii), note that 
		$$ \Omega_0+\bz=\Omega_0+p\bz+\{0,\dots,p-1\}=\Omega+\{0,\dots,p-1\}=\br.$$
	 	The sets $\Omega_0+pk+j$, with $k\in\bz$ and $j\in\{0,\dots,p-1\}$ are disjoint by (ii) and the fact that $\Omega$ tiles $\br$ by $\{0,\dots,p-1\}$. 
	 	Thus $\Omega_0$ tiles $\br$ by $\bz$. 
	 	
	 	With Theorem \ref{thf1}, we get that $\Omega_0$ has spectrum $\bz$. Also, let $\Omega_{0,k}=\Omega_0\cap[k,k+1)$ and let $I_k=\Omega_{0,k}-k\subset [0,1)$, for all $k\in\bz$. Then the sets $\{I_k:k\in\bz\}$ are disjoint (otherwise the sets $\Omega_0+k$ are not), and their union is $[0,1)$ because, for $x\in[0,1)$ there is a unique $k\in\bz$ and $x_0\in\Omega_0$ such that $x=x_0-k$. So $x_0\in\Omega_{0,k}$ and $x\in I_k$. Thus $\Omega_0$ is congruent to $[0,1)$ modulo $\bz$ and 
	 	$$|\Omega_0|=\sum_{k\in\bz}|\Omega_{0,k}|=\sum_{k\in\bz}|I_k|=1.$$
	\end{proof}

	\begin{proposition}
		\label{pr2} For $n\in\bn$ let
		$$\Omega_n:=\Omega_0+p\{-n,-n+1,\dots,n-1,n\}.$$
		\begin{enumerate}
			\item The sets $\Omega_n$ are increasing and cover $\Omega$:
			$$\cup_{n\in\bn}\Omega_n=\Omega.$$
			\item The set $\Omega_n$ tiles $\Omega$ by $p(2n+1)\bz$.
			\item The set $\Omega_n$ has spectrum 
			$$\Lambda_n:=\frac1{p(2n+1)}\left\{-n,-n+1,\dots,n\right\}+\bz.$$
		\end{enumerate}
	\end{proposition}

	\begin{proof}

	(i) is clear, from Proposition \ref{pr1}(ii). For (ii), note that the set $L_n=\{-n,\dots,n\}$ tiles $\bz$ by $(2n+1)\bz$. Therefore 
	$$\Omega_n+p(2n+1)\bz=\Omega_0+pL_n+p(2n+1)\bz=\Omega_0+p\bz,$$
	and the sets $\Omega_n+p(2n+1)k=\Omega_0+pL_n+p(2n+1)k$, $k\in\bz$ are disjoint. 
	
	For (iii), we will use the following well known lemma (see e.g., \cite{JP98a}). We include the proof for convenience. 
	
	\begin{lemma}\label{lemsp}
		Let $\mu$ be finite Radon measure on $\br$ with total measure $\mu(\br)=m$, and let $\Lambda$ be a finite or countable subset of $\br$. The measure $\mu$ has spectrum $\Lambda$ if and only if 
		\begin{equation}
			\label{eqsp}
			\sum_{\lambda\in\Lambda}|\hat\mu(t+\lambda)|^2=m^2,
		\end{equation}
		where $\hat\mu$ is the Fourier transform of the measure $\mu$,
		$$\hat\mu(t)=\int e^{-2\pi itx}\,d\mu(x),\quad(t\in\br).$$
	\end{lemma}
	
	\begin{proof}
		Suppose $\Omega$ has spectrum $\Lambda$. Then $$\{\frac{1}{\sqrt{m}}e_\lambda : \lambda\in\Lambda\}$$ is an orthonormal basis for $L^2(\mu)$. Applying the Parseval identity to the function $e_{-t}$, $t\in\br$ we obtain 
		$$m=\|e_{-t}\|^2=\sum_{\lambda\in\Lambda}\left|\ip{e_{-t}}{\frac{1}{\sqrt{m}}e_\lambda}_{L^2(\mu)}\right|^2
		=\frac{1}{m}\sum_{\lambda\in\Lambda}\left|\hat\mu(t+\lambda)\right|^2$$
		and this implies \eqref{eqsp}.
		
		For the converse, if we plug $t=-\lambda_0$ in \eqref{eqsp}, since $\hat\mu(0)=|\Omega|=m$, we get that $\hat\mu(-\lambda_0+\lambda)=0$ for $\lambda\neq \lambda_0$, and this proves that the exponentials $e_\lambda$ and $e_{\lambda_0}$ are orthogonal. With the same computation as above, \eqref{eqsp} then shows that 
		
		$$\|e_{-t}\|^2=\sum_{\lambda\in\Lambda}\left|\ip{e_{-t}}{\frac{1}{\sqrt{m}}e_\lambda}_{L^2(\mu)}\right|^2.$$
		On the right hand side we have the norm squared of the projection of $e_{-t}$ onto the closed span of the functions $\{e_\lambda :\lambda\in\Lambda\}$. This shows that $e_{-t}$ is in this closed span. By the Stone-Weierstrass theorem and by approximation by continuous compactly supported functions, the functions $e_{-t}$ span the entire space $L^2(\mu)$. It follows that $\{e_\lambda :\lambda\in\Lambda\}$ also span the entire space $L^2(\mu)$ and therefore they form an orthogonal basis for it. 
	\end{proof}

	Returning to the proof of Proposition \ref{pr2}(iii), we apply Lemma \ref{lemsp} to the Lebesgue measure on $\Omega_n$. Note that this measure is the convolution  $\mbox{Leb}_{\Omega_n}=\mbox{Leb}_{\Omega_0}\ast \Sha_{pL_n}$. Therefore its Fourier transform is 
	$$\hat\chi_{\Omega_n}(t)=\hat\chi_{\Omega_0}(t)\hat\Sha_{pL_n}(t).$$
	
	We also have
	$$\hat\Sha_{pL_n}(t)=\int e^{-2\pi itx}\,d\Sha_{pL_n}(x)=\sum_{l=-n}^ne^{-2\pi iplt},$$
	which implies that $\hat\Sha_{pL_n}$ has period $1/p$. 
	
	It is also easy to check that the measure $\Sha_{pL_n}$ has spectrum $\frac{1}{p(2n+1)}L_n$ (this is basically the discrete Fourier transform on $\bz_{2n+1}$, rescaled by $p$). Therefore, with Lemma \ref{lemsp}, 
	
	\begin{equation}
		\label{eq2.1}
			\sum_{j\in L_n}\left|\hat\Sha_{pL_n}\left(t+\frac j{p(2n+1)}\right)\right|^2=|L_n|^2=(2n+1)^2,\quad(t\in\br).
	\end{equation}

	Since $\Omega_0$ has spectrum $\bz$, we have that 
	
	\begin{equation}
		\label{eq2.2}
		\sum_{k\in\bz}\left|\hat\chi_{\Omega_0}(t+k)\right|^2=|\Omega_0|^2=1,\quad(t\in\br).
	\end{equation}

	We compute 
	
	$$\sum_{j\in L_n}\sum_{k\in\bz}\left|\hat\chi_{\Omega_n}\left(t+\frac j{p(2n+1)}+k\right)\right|^2$$$$=\sum_{j\in L_n}\sum_{k\in\bz}\left|\hat\chi_{\Omega_0}\left(t+\frac{j}{p(2n+1)}+k\right)\right|^2\left|\hat\Sha_{pL_n}\left(t+\frac{j}{p(2n+1)}+k\right)\right|^2$$
	and using the periodicity of $\hat\Sha_{pL_n}$,
	$$=\sum_{j\in L_n}\left|\hat\Sha_{pL_n}\left(t+\frac{j}{p(2n+1)}\right)\right|^2\sum_{k\in\bz}\left|\hat\chi_{\Omega_0}\left(t+\frac{j}{p(2n+1)}+k\right)\right|^2$$
	and with \eqref{eq2.2}, and then \eqref{eq2.1},
	$$=\sum_{j\in L_n}\left|\hat\Sha_{pL_n}\left(t+\frac{j}{p(2n+1)}\right)\right|^2\cdot 1=(2n+1)^2=|\Omega_n|^2.$$
	Then, with Lemma \ref{lemsp}, we conclude that $\frac{1}{p(2n+1)}L_n+\bz$ is a spectrum for $\Omega_n$. 
	\end{proof}
	
	Next we prove the first condition for $\Omega$ to be spectral, from Definition \ref{defp1}, namely that the Fourier transform is isometric from $L^2(\Omega)$ to the $L^2$-space of the pair measure, which is $p$ times the Lebesgue measure on $\left[-\frac1{2p},\frac1{2p}\right]+\bz$. The notation for this $L^2$-space is $L^2\left(\left[-\frac1{2p},\frac1{2p}\right]+\bz,pdx\right)$. 
	
	\begin{proposition}
		\label{iso}
		The Fourier transform  $$L^2(\Omega)\ni f\to \hat f|_{\left[-\frac1{2p},\frac1{2p}\right]+\bz}\in L^2\left(\left[-\frac1{2p},\frac1{2p}\right]+\bz,pdx\right)$$ is an isometry.
	\end{proposition}
	
	\begin{proof}
		Let $f$ be a $C^\infty$-function compactly supported in $\Omega$.   Then $\hat f$ is a Schwartz function. 
		
		Let $n$ be big enough so that the support of $f$ is covered by $\Omega_n=\Omega_0+p\{-n,\dots,n\}$. Since $\Omega_n$ has spectrum $\frac{1}{p(2n+1)}L_n+\bz$, we have, using the Parseval identity, 
		$$\int_{\Omega}|f(x)|^2\,dx=\int_{\Omega_n}|f(x)|^2\,dx=\sum_{j\in L_n,k\in\bz}\left|\ip{f}{\frac{1}{\sqrt{2n+1}}e_{\frac{j}{p(2n+1)}+k}}_{L^2(\Omega_n)}\right|^2$$
		
	$$=\frac{1}{2n+1}\sum_{j\in L_n,k\in\bz}\left|\hat f\left(\frac{j}{p(2n+1)}+k\right)\right|^2.$$
	This is a Riemann sum for the function $|\hat f|^2$ with sampling points $\frac{j}{p(2n+1)}+k$, $j\in L_n$, $k\in\bz$, and thus the length of the intervals of the partition is $\frac{1}{p(2n+1)}$. Since the common factor is $\frac{1}{2n+1}$ we have to adjust it by the constant $p$, and therefore the function that is integrated is actually $p|\hat f|^2$ on the set $\left[-\frac1{2p},\frac1{2p}\right]+\bz$. 
	
	Since $\hat f$ is a Schwartz function, these Riemann sums converge to the integral 
	$$\int_{\left[-\frac1{2p},\frac1{2p}\right]+\bz}|\hat f(t)|^2p\,dt.$$
	This implies that 
	$$\int_\Omega|f(x)|^2\,dx=\int_{\left[-\frac1{2p},\frac1{2p}\right]+\bz}|\hat f(t)|^2p\,dt.$$
	Since $C_0^\infty(\Omega)$ is dense in $L^2(\Omega)$, we get that the Fourier transform establishes an isometry between the two $L^2$-spaces. 
		
	\end{proof}
	
	\begin{proposition}
		\label{pr3}
		The Fourier transform from $L^2(\Omega)$ to $L^2\left(\left[-\frac1{2p},\frac1{2p}\right]+\bz, pdx\right)$ is onto. 
	\end{proposition}

	\begin{proof}
		First, note that since $\Omega_0$ is congruent to $[0,1]$ $\mod\bz$, we have for $k\in\bz$, 
		\begin{equation}
			\label{eq3.1}
				\hat\chi_{\Omega_0}(k)=\int_{\Omega_0}e^{-2\pi i kx}\,dx=\int_{[0,1]}e^{-2\pi ikx}\,dx=\left\{\begin{array}{cc}1,&k=0,\\ 0,&k\neq 0.\end{array}\right.
		\end{equation}
		
		The function $\chi_{\Omega}=\chi_{\Omega_0+p\bz}$ has period $p$. Therefore, its Fourier transform, as a tempered distribution, is
		\begin{equation}
			\label{eq3.2}
			\hat\chi_{\Omega}=\frac1p\sum_{k\in\bz}\hat\chi_{\Omega_0}(k/p)\delta_{k/p}.
		\end{equation}
		
		This follows from the next lemma:
		\begin{lemma}\label{lem3}
			Let $\tilde h$ be a function of period $p$ on $\br$ and $h=\tilde h|_{[0,p]}\in L^2[0,p]$. Represent $\tilde h$ in the Fourier basis $\frac{1}{\sqrt{p}}e_{k/p}$, $k\in\bz$:
			$$\tilde h=\sum_{k\in\bz}\frac1p\ip{h}{e_{k/p}}_{L^2[0,p]}e_{k/p}=\frac{1}{p}\sum_{k\in\bz}\hat h(k/p) e_{k/p}.$$
			Then the Fourier transform of $\tilde h$ as a tempered distribution is 
			\begin{equation}
				\label{eql3.1}
					\hat {\tilde h}=\frac1p\sum_{k\in\bz} \hat h(k/p)\delta_{k/p}.
			\end{equation}

		\end{lemma}
		
		\begin{proof}
			Let $\Lambda_{\tilde h}$ be the tempered distribution associated to $\tilde h$, 
			$$\Lambda_{\tilde h}(\varphi)=\int_{\br}\tilde h(x)\varphi(x)\,dx,\quad(\varphi\mbox{ Schwartz function}).$$
			Then, for a Schwartz function $\varphi$, 
			$$\hat\Lambda_{\tilde h}(\varphi)=\Lambda_{\tilde h}(\hat\varphi)=\frac{1}{p}\sum_{k\in\bz} \hat h(k/p)\int_{\br}e^{2\pi itk/p}\hat\varphi(t)\,dt$$
			with the Fourier Inversion Formula, 
			$$=\frac{1}{p}\sum_{k\in\bz} \hat h(k/p)\varphi(k/p),$$
			and \eqref{eql3.1} follows. 
			The sums are convergent because $\sum |\hat h(k/p)|^2<\infty$ and $\varphi,\hat\varphi$ are Schwartz functions, thus they have rapid decay.

		\end{proof}
		
		Take now some $C_0^\infty$-function $g$ supported on one of the intervals $\left[-\frac1{2p},\frac1{2p}\right]+k_0$ for some $k_0\in\bz$, and let $\varphi$ be its inverse Fourier transform; it is a Schwartz function. Let
		$$f=p\chi_\Omega\cdot \varphi.$$
		Then 
		$$\hat f(t)=p\hat\chi_\Omega\ast\hat \varphi(t)=\left(\sum_{k\in\bz}\hat\chi_{\Omega_0}(k/p)\delta_{k/p}\right)\ast g (t)=\sum_{k\in\bz}\hat\chi_{\Omega_0}(k/p)g(t-k/p).$$
		We need to restrict $\hat f$ to $\left[-\frac1{2p},\frac1{2p}\right]+\bz$. Note that $g(t-k/p)$ is supported on $\left[-\frac1{2p},\frac1{2p}\right]+\frac kp+k_0$, and these sets are disjoint for different $k$'s. Thus, when we restrict $\hat f$ to $\left[-\frac1{2p},\frac1{2p}\right]+\bz$, we will keep only the terms with $k=pl$ for $l\in\bz$:
		$$\hat f\cdot \chi_{\left[-\frac1{2p},\frac1{2p}\right]+\bz}(t)=\sum_{l\in\bz}\hat\chi_{\Omega_0}(pl/p)g(t-pl/p).$$
		But, by \eqref{eq3.1}, $\hat\chi_{\Omega_0}(l)=0$ unless $l=0$, and therefore the only term left is $g(t)$. Thus 
		$$\hat f|_{\left[-\frac1{2p},\frac1{2p}\right]+\bz}=g.$$
		Since the linear combinations of functions like $g$ are dense in $L^2(\left[-\frac1{2p},\frac1{2p}\right]+\bz,pdx)$ and since we proved that the map is an isometry, it follows that the map is onto.
	\end{proof}
	
	Propositions \ref{iso} and \ref{pr3} show the direct implication: if $\Omega$ tiles with $\{0,1,\dots,p-1\}$ then it is spectral with the desired pair measure. 
	
	We move on now to the converse and assume that $\Omega$ is spectral with the given pair measure.
	
	\begin{proposition}
		\label{pr4}
		Suppose $\Omega$ is spectral with pair measure $p\Leb_{\left[-\frac1{2p},\frac1{2p}\right]+\bz}$. Let $f,g$ be in $L^2(\Omega)$. Then $\Per(\hat f\cj{\hat g})$ has period $1/p$.
	\end{proposition}	
	
	\begin{proof}
		Let $a\in\br$. We have 
		$$\ip{f}{g}_{L^2(\Omega)}=\ip{e_af}{e_ag}_{L^2(\Omega)}=\int_{\left[-\frac1{2p},\frac1{2p}\right]+\bz}\hat f(t-a)\cj{\hat g}(t-a)p\,dt$$
		$$=\sum_{k\in\bz}\int_{\left[-\frac1{2p},\frac1{2p}\right]}\hat f(t-a+k)\cj{\hat g}(t-a+k)p\,dt$$$$=\int_{\left[-\frac1{2p},\frac1{2p}\right]}\Per(\hat f\cj{\hat g})(t-a)p\,dt=\int_{\left[-\frac1{2p},\frac1{2p}\right]-a}\Per(\hat f\cj{\hat g})(t)p\,dt.$$
		
		Let $h(t)=\Per(\hat f\cj{\hat g})(t)$. Then the function 
		$$a\to \int_{\left[-\frac1{2p},\frac1{2p}\right]-a} h(t)\,dt$$
		is constant. 
		
		Let $-a_1=\frac{1}{2p}+b+\epsilon$, $-a_2=\frac1{2p}+b$, for $b\in\br$ and $\epsilon>0$. We have
		$$0=\int_{\left[-\frac1{2p},\frac1{2p}\right]-a_1} h(t)\,dt-\int_{\left[-\frac1{2p},\frac1{2p}\right]-a_2} h(t)\,dt=\int_{[b+\epsilon,b+\frac1p+\epsilon]} h(t)\,dt-\int_{[b,b+\frac1p]} h(t)\,dt$$$$=\int_{[b+\frac1p,b+\frac1p+\epsilon]} h(t)\,dt-\int_{[b,b+\epsilon]} h(t)\,dt.$$
		Divide by $\epsilon$ and, with the Lebesgue Differentiation Theorem, we get that for almost every $b\in\br$, 
		$$h(b+1/p)=h(b),$$
		thus $\Per(\hat f\cj{\hat g})$ has period $\frac1p$. 
	\end{proof}

	\begin{proposition}\label{pr5}
			Suppose $\Omega$ is spectral with pair measure $p\Leb_{\left[-\frac1{2p},\frac1{2p}\right]+\bz}$. Then $\Omega\cap(\Omega+j)=\ty$ up to measure zero, for all $j\in\bz$, $j\not\equiv 0\mod p$.
	\end{proposition}
	
	\begin{proof}
		Let $f,g\in L^2(\Omega)$. We compute, with Plancherel's identity,
		$$\ip{f}{T_{j}g}_{L^2(\br)}=\int_{\br}{\hat f(t)}\cj{e^{-2\pi i jt}\hat g(t)}\,dt=\sum_{l\in\bz}\int_{\left[-\frac1{2p},\frac1{2p}\right]}\hat f(t+l/p)e^{2\pi i j(t+l/p)}\cj{\hat g}(t+l/p)\,dt$$
		$$=\int_{\left[-\frac1{2p},\frac1{2p}\right]}\sum_{i=0}^{p-1}\sum_{m\in\bz}\hat f(t+i/p+m)\cj{\hat g}(t+i/p+m)e^{2\pi ij(t+i/p)}\,dt$$
		$$=\int_{\left[-\frac1{2p},\frac1{2p}\right]}\sum_{i=0}^{p-1}e^{2\pi i j(t+i/p)}\Per(\hat f \cj{\hat g})(t+i/p)\,dt$$
		with Propositon \ref{pr4},
		$$=\int_{\left[-\frac1{2p},\frac1{2p}\right]}\Per(\hat f \cj{\hat g})(t)e^{2\pi ijt}\sum_{i=0}^{p-1}e^{2\pi i ji/p}\,dt=0,$$
		since $j\not\equiv0(\mod p)$.
		
		Suppose now that $\Omega\cap (\Omega+j)$ has positive measure. Then take $E\subset\Omega\cap(\Omega+j)$ of finite measure. Let $f=\chi_{E}$ and $g=\chi_{E-j}$. Then $f,g\in L^2(\Omega)$ and $T_j g=f$ Thus $0=\ip{f}{T_jg}_{L^2(\br)}=\|f\|_{L^2(\br)}^2>0$, a contradiction . 
	\end{proof}
	
	We want to show that $\Omega$ tiles $\br$ by $\{0,\dots,p-1\}$. We will show next that $\Omega$ is contained in a larger set $\tilde \Omega$ which tiles $\br$ by $\{0,\dots,p-1\}$.

The set $\Omega'=\Omega+p\bz$ is clearly invariant under translations by $p\bz$ and we show that 
	\begin{equation}
		\label{eq5.1}
		\{\Omega'+j : j\in\{0,1,\dots,p-1\}\}\mbox{ are disjoint sets}.
	\end{equation}

	If not, there exist $k_1,k_2\in\bz$, $j_1,j_2\in\{0,\dots, p-1\}$ such that 
	$(\Omega+pk_1+j_1)\cap (\Omega+pk_2+j_2)\neq\ty$. We can assume  $j_1>j_2$ and then $j_1-j_2\in\{0,\dots,p-1\}$ and 
$(\Omega+p(k_1-k_2)+(j_1-j_2))\cap \Omega\neq \ty$, which contradicts Proposition \ref{pr5}.

	\begin{lemma}\label{lem6}
			Define  
		$$\Omega_0:=\Omega'\cap[0,p).$$
		\begin{enumerate}
			\item The sets $\{\Omega_0+k : k\in \bz\}$ are disjoint.
			\item $\Omega'=\Omega+p\bz=\Omega_0+p\bz$.
		\end{enumerate}		
	\end{lemma}
	
	\begin{proof}
		Consider two sets $\Omega_0+pl+j$ and $\Omega_0+pl'+j'$ with $l,l'\in\bz$ and $j,j'\in\{0,\dots,p-1\}$. 
		If $j\neq j'$ then the intersection of these two sets is contained in $(\Omega'+j)\cap (\Omega'+j')=\ty$. 
		If $j=j'$ and the intersection of the two sets is nonempty, then $(\Omega_0+pl)\cap(\Omega_0+pl')\neq \ty$ and this is impossible because $\Omega_0$ is contained in $[0,p)$. This proves (i). 
		
		For (ii), we clearly have $\Omega_0+p\bz\subset \Omega+p\bz$. For the reverse inclusion, let $x\in \Omega'$. Then, there exists $k\in\bz$ such that $x-pk\in[0,p)$. Since $\Omega'$ is invariant under translations by $p\bz$, it follows that $x-pk\in\Omega'\cap[0,p)=\Omega_0$ and thus $x\in \Omega_0+pk$.
	\end{proof}
	
	\begin{lemma}
		\label{lem7}
		Define $R:=[0,1)\setminus ((\Omega_0+\bz)\cap [0,1))$. $\Omega_0\cup R$ tiles $\br$ by $\bz$ and $\tilde\Omega:=(\Omega_0\cup R)+p\bz$ tiles $\br$ by $\{0,\dots,p-1\}$ and contains $\Omega$.
	\end{lemma}
	
	\begin{proof}
		Note that $\Omega_0+\bz$ and $R+\bz$ are disjoint. Otherwise, there exist $x_0\in\Omega_0$, $k\in\bz$, $r\in R$ and $l\in \bz$ such that $x_0+k=r+l$. So $r=x_0+(k-l)\in (\Omega_0+\bz)\cap[0,1)$, which contradicts the definition of $R$. The sets $\{R+k: k\in\bz\}$ are disjoint because $R$ is contained in $[0,1)$. The sets $\{\Omega_0+k : k\in\bz\}$ are disjoint, by Lemma \ref{lem6}. It follows that the sets $\{(\Omega_0\cup R)+k : k\in\bz\}$ are disjoint. Also, 
		$$(\Omega_0\cup R)+\bz=(\Omega_0+\bz)\cup (R+\bz)=((\Omega_0+\bz)\cup R)+\bz\supseteq[0,1)+\bz=\br.$$
		This shows that $\Omega_0\cup R$ tiles $\br$ by $\bz$.
		
		For $\tilde \Omega=(\Omega_0\cup R)+p\bz$, since $\Omega_0\cup R$ tiles $\br$ with $\bz$, we get that the sets $\{\tilde\Omega+j=(\Omega_0\cup R)+j+p\bz : j\in\{0,\dots,p-1\}\}$ are disjoint, and 
		$$\tilde\Omega+\{0,\dots,p-1\}=(\Omega_0\cup R)+p\bz+\{0,\dots,p-1\}=(\Omega_0\cup R)+\bz=\br.$$
		Thus, $\tilde\Omega$ tiles $\br$ by $\{0,\dots,p-1\}$.
		
		By Lemma \ref{lem6}, $\Omega\subseteq\Omega_0+p\bz\subseteq\tilde\Omega$.

	\end{proof}
	
	\begin{lemma}
		\label{lem8}
		$\Omega=\tilde\Omega$ and therefore $\Omega$ tiles $\br$ by $\{0,\dots,p-1\}$.
	\end{lemma}
	
	\begin{proof}
		Suppose the inclusion $\Omega\subset\tilde\Omega$ is proper, so the complement $\tilde\Omega\setminus\Omega$ has positive Lebesgue measure. 
		
		Let $g\in L^2(\tilde\Omega\setminus\Omega)$, $g\neq 0$ and let $f\in L^2(\Omega)$. Then $f\perp g$ in $L^2(\tilde\Omega)$. Since $\tilde\Omega$ tiles $\br$ by $\{0,\dots,p-1\}$, by the direct implication of Theorem \ref{thm},  we know that the Fourier transform
		$$ L^2(\tilde\Omega)\ni f\to \hat f|_{\left[-\frac1{2p},\frac1{2p}\right]+\bz}=:\mathcal F_r f\in   L^2\left(\left[-\frac1{2p},\frac1{2p}\right]+\bz,p\,dx\right)$$
		is a surjective isometry. Therefore $\mathcal F_r f\perp \mathcal F_r g$, and also $\mathcal F_rg\neq 0$. 
		
		On the other hand, by hypothesis, $\mathcal F_r$ is a surjective isometry from $L^2(\Omega)$ to the same range $  L^2\left(\left[-\frac1{2p},\frac1{2p}\right]+\bz,p\,dx\right)$. Thus $\{\mathcal F_rf : f\in L^2(\Omega)\}$ is the entire range $  L^2\left(\left[-\frac1{2p},\frac 1{2p}\right]+\bz,p\,dx\right)$; therefore, $\mathcal F_rg$ is orthogonal to the entire range, and thus it is orthogonal to itself, so it has to be zero, a contradiction.

	\end{proof}
	
	In conclusion, $\Omega=\tilde\Omega$ and thus $\Omega$ tiles $\br$ with $\{0,\dots,p-1\}$, and this shows the converse implication in our Theorem \ref{thm}.
	\end{proof}

	{\bf Conflict of interest.} On behalf of all authors, the corresponding author states that there is no conflict of interest.
	\bibliographystyle{alpha}	
	
	{\bf Data Availability Statement.} No new data were created or analyzed in this study. Data sharing is not applicable to this article.
	
	{\bf Funding.} No funds, grants, or other support were received during the preparation of this manuscript.
	
	{\bf Author Contribution.} All authors contributed equally to the study conception, design, data collection, analysis, and manuscript preparation. Both authors read and approved the final manuscript.
	
	
	
	\bibliography{eframes}
	
	\end{document}